\documentclass{article}

\usepackage{amsmath,amsthm,amssymb,amsfonts}
\usepackage{verbatim}

\usepackage{stmaryrd} 
\usepackage{hyperref}

\usepackage[colorinlistoftodos]{todonotes}
\usepackage{tikz}

\usepackage{tkz-berge, tkz-graph}

\tikzset{vertex/.style={circle,draw,fill,inner sep=0pt,minimum size=1mm}}

\newtheorem{thm}{Theorem}
\newtheorem{lem}[thm]{Lemma}

\newtheorem{prop}[thm]{Proposition}
\newtheorem{cor}[thm]{Corollary}

\newtheorem{assert}[thm]{Assertion}

\newtheorem{remarks}[thm]{Remark}

\newtheorem{definition}[thm]{Definition}

\newtheorem{exl}[thm]{Example}

\numberwithin{thm}{section}

\newcommand{\adj}{\leftrightarrow}

\newcommand{\adjeq}{\leftrightarroweq}

\def\Z{{\mathbb Z}}
\def\N{{\mathbb N}}

\def\R{{\mathbb R}}

\begin{document}

\title{Remarks on Fixed Point Assertions in Digital Topology, 4}

\author{Laurence Boxer
\thanks{Department of Computer and Information Sciences, Niagara University, NY 14109, USA; and Department of Computer Science and Engineering, State University of New York at Buffalo
email: boxer@niagara.edu}
}
\date{ }
\maketitle
\begin{abstract}
We continue the work of~\cite{BxSt19,Bx19,Bx19-3}, in which we discuss published assertions 
that are incorrect or incorrectly proven; that are severely limited or reduce to triviality;
or that we improve upon.


\end{abstract}

\section{Introduction}
As stated in~\cite{Bx19}:
\begin{quote}
The topic of fixed points in digital topology has drawn
much attention in recent papers. The quality of
discussion among these papers is uneven; 
while some assertions have been correct and interesting, others have been incorrect, incorrectly proven, or reducible to triviality.
\end{quote}
Paraphrasing~\cite{Bx19} slightly: 
in~\cite{BxSt19,Bx19,Bx19-3}, we have discussed many shortcomings in earlier papers and have offered
corrections and improvements. We continue this work in the current paper.

A common theme among many weak papers concerning fixed points in digital topology
is the use of a ``digital metric space" (see section~\ref{DigMetSp} for its definition).
This seems to be a bad idea.
\begin{itemize}
    \item Nearly all correct nontrivial published assertions concerning digital metric 
          spaces use either the adjacency of the digital image or the metric, but not both. Thus the notion of a digital metric space does not seem natural.
    \item If $X$ is finite (as in the ``real world") or the metric $d$ 
          is a common metric such as any $\ell_p$ metric, then $(X,d)$
          is discrete as a topological space, hence not very interesting.
    \item Many of the published assertions concerning digital metric spaces
          mimic analogues for connected subsets of Euclidean~$\R^n$. Often,
          the authors neglect important differences between the topological 
          space $\R^n$ and digital images, resulting in assertions that are 
          incorrect, trivial, or trivial when restricted to conditions that
          many others regard as essential. E.g., in many cases, functions that
          satisfy fixed point assertions must be constant or fail to be
          digitally continuous~\cite{BxSt19,Bx19,Bx19-3}.
\end{itemize}
This paper continues the work of~\cite{BxSt19,Bx19,Bx19-3} in discussing 
shortcomings of published assertions concerning fixed points in digital topology.

\section{Preliminaries}
We use $\N$ to represent the natural numbers,
$\Z$ to represent the integers, and $\R$ to represent the reals.

A {\em digital image} is a pair $(X,\kappa)$, where $X \subset \Z^n$ 
for some positive integer $n$, and $\kappa$ is an adjacency relation on $X$. 
Thus, a digital image is a graph.
In order to model the ``real world," we usually take $X$ to be finite,
although there are several papers that consider
infinite digital images. The points of $X$ may be 
thought of as the ``black points" or foreground of a 
binary, monochrome ``digital picture," and the 
points of $\Z^n \setminus X$ as the ``white points"
or background of the digital picture.

\subsection{Adjacencies, connectedness, continuity, fixed point}
In a digital image $(X,\kappa)$, if
$x,y \in X$, we use the notation
$x \adj_{\kappa}y$ to
mean $x$ and $y$ are $\kappa$-adjacent; we may write
$x \adj y$ when $\kappa$ can be understood. 
We write $x \adjeq_{\kappa}y$, or $x \adjeq y$
when $\kappa$ can be understood, to
mean 
$x \adj_{\kappa}y$ or $x=y$.

The most commonly used adjacencies in the study of digital images 
are the $c_u$ adjacencies. These are defined as follows.
\begin{definition}
Let $X \subset \Z^n$. Let $u \in \Z$, $1 \le u \le n$. Let 
$x=(x_1, \ldots, x_n),~y=(y_1,\ldots,y_n) \in X$. Then $x \adj_{c_u} y$ if 
\begin{itemize}
    \item for at most $u$ distinct indices~$i$,
    $|x_i - y_i| = 1$, and
    \item for all indices $j$ such that $|x_j - y_j| \neq 1$ we have $x_j=y_j$.
\end{itemize}
\end{definition}

\begin{definition}
{\rm \cite{Rosenfeld}}
A digital image $(X,\kappa)$ is
{\em $\kappa$-connected}, or just {\em connected} when
$\kappa$ is understood, if given $x,y \in X$ there
is a set $\{x_i\}_{i=0}^n \subset X$ such that
$x=x_0$, $x_i \adj_{\kappa} x_{i+1}$ for
$0 \le i < n$, and $x_n=y$.
\end{definition}

\begin{definition}
{\rm \cite{Rosenfeld, Bx99}}
Let $(X,\kappa)$ and $(Y,\lambda)$ be digital
images. A function $f: X \to Y$ is 
{\em $(\kappa,\lambda)$-continuous}, or
{\em $\kappa$-continuous} if $(X,\kappa)=(Y,\lambda)$, or
{\em digitally continuous} when $\kappa$ and
$\lambda$ are understood, if for every
$\kappa$-connected subset $X'$ of $X$,
$f(X')$ is a $\lambda$-connected subset of $Y$.
\end{definition}

\begin{thm}
{\rm \cite{Bx99}}
A function $f: X \to Y$ between digital images
$(X,\kappa)$ and $(Y,\lambda)$ is
$(\kappa,\lambda)$-continuous if and only if for
every $x,y \in X$, if $x \adj_{\kappa} y$ then
$f(x) \adjeq_{\lambda} f(y)$.
\end{thm}

\begin{thm}
\label{composition}
{\rm \cite{Bx99}}
Let $f: (X, \kappa) \to (Y, \lambda)$ and
$g: (Y, \lambda) \to (Z, \mu)$ be continuous 
functions between digital images. Then
$g \circ f: (X, \kappa) \to (Z, \mu)$ is continuous.
\end{thm}

We use $1_X$ to denote the identity function on $X$, and $C(X,\kappa)$ for
the set of functions $f: X \to X$ that are $\kappa$-continuous.

A {\em fixed point} of a function $f: X \to X$ is a
point $x \in X$ such that $f(x) = x$. 

Functions $f,g: X \to X$ are {\em commuting} if $f(g(x))=g(f(x))$ for all $x \in X$.

\subsection{Digital metric spaces}
\label{DigMetSp}
A {\em digital metric space}~\cite{EgeKaraca15} is a triple
$(X,d,\kappa)$, where $(X,\kappa)$ is a digital image and $d$ is a metric on $X$. 
We are not convinced that
this is a notion worth developing; under conditions 
in which a digital image models a ``real world" image, 
$X$ is finite or $d$ is (usually) an $\ell_p$ metric, so that 
$(X,d)$ is discrete as a topological space. Typically,
assertions in the literature do not make use of both
$d$ and $\kappa$,
so that this notion has an artificial feel. E.g., for a
discrete topological space, all self-maps are continuous,
although on digital images, many self-maps are not digitally continuous.

We say a sequence $\{x_n\}_{n=0}^{\infty}$ is {\em eventually constant} if for some $m>0$, $n>m$
implies $x_n=x_m$.

\begin{prop}
{\em \cite{Han16,BxSt19}}
\label{eventuallyConst}
Let $(X,d, \kappa)$ be a digital metric space. If for some $a>0$ and all
distinct $x,y \in X$ we have $d(x,y) > a$, then any Cauchy sequence in $X$
is eventually constant, and $(X,d)$ is a complete metric space.
\end{prop}

Note that the hypotheses of Proposition~\ref{eventuallyConst} are satisfied
if $X$ is finite or if $d$ is an $\ell_p$ metric.

\subsection{Common conditions, limitations, and trivialities}
In this section, we state results that limit or trivialize several of the assertions discussed later
in this paper.

Although there are papers that discuss infinite digital images, a ``real world" digital
image is a finite set. Further, most authors writing about a digital metric space
choose their metric from the Euclidean metric, the Manhattan metric, or some other $\ell_p$ metric. 

Other frequently used conditions:
\begin{itemize}
    \item The adjacencies most often used in the digital topology literature are the $c_u$ adjacencies.
    \item Functions that attract the most interest in the digital topology literature are digitally continuous.
\end{itemize}
Thus, the use of $c_u$-adjacency and the continuity assumption (as well as the assumption
of an $\ell_p$ metric) in the following
Proposition~\ref{shrinkage} should not be viewed as major restrictions.
The following is taken from the proof of Remark~5.2 of~\cite{BxSt19}.
\begin{prop}
\label{shrinkage}
Let $X$ be $c_u$-connected. Let $T \in C(X,c_u)$. Let $d$ be an
$\ell_p$ metric on $X$, and $0< \alpha < \frac{1}{u^{1/p}}$.
Let $S: X \to X$ such that $d(S(x), S(y)) \le \alpha d(T(x),T(y))$ for all $x,y \in X$.
Then~$S$ must be a constant function.
\end{prop}

We discuss several results in the literature that are trivial when Proposition~\ref{shrinkage} is
applicable.

The notions of convergent sequence and complete digital metric space are often trivial, e.g., if the
digital image is finite or the metric used is an $\ell_p$ metric, as noted in the following.

\begin{thm}
{\rm \cite{Han16}}
\label{HanTrivialization}
Let $(X,d,\kappa)$ be a digital metric space. If there is a constant $c>0$ such that
for all distinct $x,y \in X$ we have $d(x,y) > c$, then any Cauchy sequence in $X$ is
eventually constant, and $(X,d)$ is a complete metric space.
\end{thm}

\begin{remarks}
Notice that any $\ell_p$ metric satisfies the hypothesis of
Theorem~\ref{HanTrivialization}. Therefore,
in the rest of this paper, wherever we assume $d$ is an
$\ell_p$ metric, we could use instead the more general
hypothesis of Theorem~\ref{HanTrivialization}.
\end{remarks}

\begin{cor}
\label{converge}
Let $(X,d,\kappa)$ be a digital metric space, where $X$ is finite or $d$ is as in Theorem~\ref{HanTrivialization}.
Let $f \in C(X,\kappa)$. Then if $\{x_n\}_{n=1}^{\infty} \subset X$ and 
$\lim_{n \to \infty} x_n = x_0 \in X$, then for almost all $n$, $f(x_n) = f(x_0)$.
\end{cor}

Other choices of $(X,d)$ need not lead to the conclusion of Corollary~\ref{converge},
as shown by the following example.

\begin{exl}
\label{nonStdMetric}
Let $X = \N \cup \{0\}$, 
\[ d(x,y) = \left \{ \begin{array}{ll}
            0 & \mbox{if } x=0=y; \\
            1/x & \mbox{if } x \neq 0 = y; \\
            1/y & \mbox{if } x = 0 \neq y; \\
            |1/x - 1/y| & \mbox{if } x \neq 0 \neq y.
            \end{array} \right .
\]
Then $d$ is a metric, and $\lim_{n \to \infty} d(n,0) = 0$. However,
the function $f(n)=n+1$ satisfies $f \in C(X,c_1)$ and
\[ \lim_{n \to \infty} d(0, f(n)) = 0 \neq f(0).
\]
\end{exl}

\begin{proof}
Example~2.10 of~\cite{BxSt19} notes that $d$ has the properties of a metric for values of $X \setminus \{0\}$.
Since clearly $d(0,0)=0$ and $d(x,0)=d(0,x)$, we must show that the triangle inequality holds when $0$
is one of the points considered. We have the following.
\begin{itemize}
    \item If $x,y>0$ and $1/x \ge 1/y$ then $d(0,y) = 1/y \le 1/x + |1/x - 1/y| = d(0,x) + d(x,y)$.
    \item If $x,y>0$ and $1/x < 1/y$ then $d(0,y) = 1/y = 1/x + |1/x - 1/y| = d(0,x) + d(x,y)$.
    \item Similarly, if $x,y>0$ then $d(x,0) \le d(x,y) + d(y,0)$.
    \item If $x,y>0$ then $d(x,y)=|1/x - 1/y| < 1/x + 1/y = d(x,0)+d(0,y)$.
\end{itemize}
Thus, the triangle inequality is satisfied.

Note $f \in C(X,c_1)$, $\lim_{n \to \infty}d(n,0) = 0$, and
$\lim _{n \to \infty}d(f(n), f(0)) = 1$.
\end{proof}

\section{Compatible maps and weakly compatible maps}
\label{compatSection}
The papers~\cite{Jyoti-Rani-Rani17b,Rani-Jyoti-Rani, EgeEtal}
discuss common fixed points of compatible and weakly compatible maps 
(the latter also know as ``coincidentally commuting") and related notions.

\subsection{Compatibility - definition and basic properties}
\begin{definition}
{\em \cite{DalalEtal}}
\label{compatDef}
Suppose $S$ and $T$ are self-maps on a digital metric space $(X,d,\kappa)$.
Consider $\{x_n\}_{n=1}^{\infty} \subset X$ such that
\begin{equation}
    \label{compatSeq}
    \lim_{n \to \infty} S(x_n) = \lim_{n \to \infty} T(x_n) = t \in X.
\end{equation}
If every sequence satisfying~{\rm(\ref{compatSeq})} also satisfies
$\lim_{n \to \infty} d(S(T(x_n)), T(S(x_n))) = 0$, then $S$ and $T$ are
{\em compatible maps}.
\end{definition}

\begin{prop}
\label{commuteImpliesCompatible}
Let $f,g: X \to X$ be commuting functions on a digital metric space $(X,d,\kappa)$.
Then $f$ and $g$ are compatible.
\end{prop}

\begin{proof}
Since $f$ and $g$ are commuting, the assertion is immediate.
\end{proof}

\begin{definition}
{\rm \cite{Jyoti-Rani-Rani17b}}
\label{EAdef}
Let $S$ and $T$ be self-maps on a digital metric space $(X,d,\kappa)$.
The pair $(S,T)$ satisfies the property {\em E.A.} if there exists a sequence
$\{x_n\}_{n=1}^{\infty} \subset X$ that satisfies~(\ref{compatSeq}).
\end{definition}

\begin{definition}
{\rm \cite{Jyoti-Rani-Rani17b}}
\label{CLRTdef}
Let $S$ and $T$ be self-maps on a digital metric space $(X,d,\kappa)$.
The pair $(S,T)$ satisfies the property {\em common limit in the range of $T$ (CLRT)}  
if there exists a sequence $\{x_n\}_{n=1}^{\infty} \subset X$ that satisfies
\[ \lim_{n \to \infty} S(x_n) = \lim_{n \to \infty} T(x_n) = T(x) \mbox{ for some } x \in X.
\]
\end{definition}

\begin{prop}
\label{EAequivCLRT}
Let $S$ and $T$ be self-maps on a digital metric space $(X,d,\kappa)$.
\begin{itemize}
    \item Suppose the pair $(S,T)$ satisfies the property CLRT. Then the pair $(S,T)$ 
          satisfies the property E.A.
    \item Suppose the pair $(S,T)$ satisfies the property E.A. If $X$ is finite,
          then $(S,T)$ satisfies the property CLRT.
\end{itemize}
\end{prop}

\begin{proof}
Suppose the pair $(S,T)$ satisfies the property CLRT. It is trivial that 
the pair $(S,T)$ satisfies the property E.A.

Suppose the pair $(S,T)$ satisfies the property E.A. Let 
$\{x_n\}_{n=1}^{\infty} \subset X$ satisfy~(\ref{compatSeq}). Since $X$ is finite,
there is a subsequence $\{x_{n_i}\}$ such that $x_{n_i}$ is eventually constant; say,
$x_{n_i}$ is eventually equal to $x \in X$. Hence $T(x_n)$ is eventually
$\lim_{n_i \to \infty}T(x_{n_i}) = T(x)$. Thus $(S,T)$ satisfies the property CLRT.
\end{proof}

\subsection{Variants on compatibility}
In classical topology and real analysis, there are many papers that study
variants of compatible (as defined above) functions.
Several authors have studied analogs of these variants in digital topology.
Often, the variants turn out to be equivalent.

\begin{definition}
\label{WeaklyCompatDef}
{\rm \cite{Dalal17}}
Let $S,T: X \to X$. Then $S$ and $T$ are
{\em weakly compatible} or 
{\em coincidentally commuting} if, for every 
$x \in X$ such that $S(x)=T(x)$ we have 
$S(T(x)) = T(S(x))$.
\end{definition}

\begin{thm}
\label{compatWeakCompat}
Let $S,T: X \to X$. Compatibility implies weak compatibility; and if $X$ is 
finite, weak compatibility implies compatibility.
\end{thm}

\begin{proof}
Suppose $S$ and $T$ are compatible. We show they are weakly compatible 
as follows. Let $S(x)=T(x)$ for some $x \in X$. Let
$x_n=x$ for all $n \in \N$. Then 
\[ lim_{n \to \infty} S(x_n) = S(x) = T(x) = lim_{n \to \infty} T(x_n).
\]
By compatibility,
\[ 0 = \lim_{n \to \infty} d(S(T(x_n)), T(S(x_n))) = d(S(T(x)), T(S(x))).
\]
Thus, $S$ and $T$ are weakly compatible.

Suppose $S$ and $T$ are weakly compatible and $X$ is finite. We show $S$ and $T$ are compatible 
as follows.
Let $\{x_n\}_{n=1}^{\infty} \subset X$ such that
\[\lim_{n \to \infty} S(x_n) = \lim_{n \to \infty} T(x_n) = t \in X.
\]
Theorem~\ref{HanTrivialization} yields that for almost all~$n$, $S(x_n)=T(x_n) = t$. 
Since $X$ is finite, there is an infinite subsequence $\{x_{n_i}\}$ of
$\{x_n\}_{n=1}^{\infty}$ such that $x_{n_i} = y \in X$, hence $S(y)=T(y)$. 
Therefore, for almost all~$n$ and almost all $n_i$, weak compatibility implies
\[ S(T(x_n)) = S(T(x_{n_i})) = S(T(y)) = T(S(y)) = T(S(x_{n_i})) = T(S(x_n)).
\]
It follows that $S$ and $T$ are compatible.
\end{proof}

We have the following, in which we
restate~(\ref{compatSeq}) for convenience.

\begin{definition}
\label{compatVariantsDef}
Suppose $S$ and $T$ are self-maps on a digital metric space $(X,d,\kappa)$.
Consider $\{x_n\}_{n=1}^{\infty} \subset X$ such that
\begin{equation}
    \label{repeatedCompatSeq}
    \lim_{n \to \infty} S(x_n) = \lim_{n \to \infty} T(x_n) = t \in X.
\end{equation}
\begin{itemize}
    \item $S$ and $T$ are {\em compatible of type A}~{\rm \cite{DalalEtal}}
           if every sequence satisfying~{\rm(\ref{repeatedCompatSeq})} also satisfies
          \[ \lim_{n \to \infty} d(S(T(x_n)), T(T(x_n))) = 0 = 
          \lim_{n \to \infty} d(T(S(x_n)), S(S(x_n))).
          \]
    \item $S$ and $T$ are {\em compatible of type B}~{\rm \cite{EgeEtal}}
           if every sequence satisfying~{\rm(\ref{repeatedCompatSeq})} also satisfies
          \[ \lim_{n \to \infty} d(S(T(x_n)), T(T(x_n))) \le \]
         \begin{equation} 
         \label{compatB1stIneq}
         1/2~[\lim_{n \to \infty} d(S(T(x_n)), S(t)) + d(S(t), S(S(x_n)))]
          \end{equation}
          and
         \[
         \lim_{n \to \infty} d(T(S(x_n)), S(S(x_n))) \le \]
         \begin{equation} 
         \label{compatB2ndIneq}
          1/2~[\lim_{n \to \infty} d(T(S(x_n)), T(t)) + d(T(t), T(T(x_n)))].
          \end{equation}
          {\em Note this is a correction of the definition as stated in~\cite{EgeEtal},
           where the inequality here given as~(\ref{compatB2ndIneq}) 
           uses a left side equivalent to}
           \[ \lim_{n \to \infty} d(T(S(x_n)), T(T(x_n))) \mbox{ instead of }
           \lim_{n \to \infty} d(T(S(x_n)), S(S(x_n))). \]
           {\em The version we have stated
           is the version used in proofs of~\cite{EgeEtal} and corresponds to the version
           of~\cite{PathakKahn} that inspired the definition of~\cite{EgeEtal}. 
           }
        \item $S$ and $T$ are {\em compatible of type C}~{\rm \cite{EgeEtal}}
           if every sequence satisfying~{\rm(\ref{repeatedCompatSeq})} also satisfies
           \[ \lim_{n \to \infty} d(S(T(x_n)), T(T(x_n))) \le
           \]
           \begin{equation}
               \label{compatC1stIneq}
            1/2~ \left [ \begin{array}{c}\lim_{n \to \infty} d(S(T(x_n)),S(t)) + 
           \lim_{n \to \infty} d(S(t), S(S(x_n))) + \\ \lim_{n \to \infty} d(S(t), T(T(x_n))) \end{array} \right ]
           \end{equation}
           and
           \[
               \lim_{n \to \infty} d(T(S(x_n)), S(S(x_n))) \le 
           \]
           \begin{equation}
               \label{compatC2ndIneq} 1/2~ \left [ \begin{array}{c} \lim_{n \to \infty} d(T(S(x_n)), T(t)) +
            \lim_{n \to \infty} d(T(t), T(T(x_n))) + \\ \lim_{n \to \infty} d(T(t), S(S(x_n))) \end{array} \right ].
           \end{equation}
        \item $S$ and $T$ are {\em compatible of type P}~{\rm \cite{DalalEtal}}
           if every sequence satisfying~{\rm(\ref{repeatedCompatSeq})} also satisfies
           \[ \lim_{n \to \infty} d(S(S(x_n)), T(T(x_n))) = 0.
           \]
\end{itemize}
\end{definition}

We augment Theorem~\ref{compatWeakCompat} with the following.

\begin{thm}
\label{manyEquivs}
Let $(X,d,\kappa)$ be a digital metric space, where $X$ is finite or $d$ is
an $\ell_p$ metric. Let $S,T: X \to X$. The following are equivalent.
\begin{itemize}
    \item $S$ and $T$ are compatible.
    \item $S$ and $T$ are compatible of type A.
    \item $S$ and $T$ are compatible of type B.
    \item $S$ and $T$ are compatible of type C.
    \item $S$ and $T$ are compatible of type P.
\end{itemize}
\end{thm}

\begin{proof}
The equivalence of compatible, compatible of type A, and compatible of type P was
shown in Theorem~3.3 of~\cite{Bx19}.

Compatible of type A implies compatible of type B, by Proposition~4.7
of~\cite{EgeEtal}.

We show compatible of type B implies compatible, as follows. Let $S$ and $T$ be
compatible of type B. Let $\{x_n\}_{n=1}^{\infty} \subset X$ 
satisfy~(\ref{repeatedCompatSeq}). By Proposition~\ref{eventuallyConst}, 
$S(x_n)=t=T(x_n)$ for almost all~$n$. From~(\ref{compatB1stIneq}) we have
\[ d(S(t), T(t)) \le 1/2~[d(S(t),S(t)) + d(S(t),S(t))] = 0, \]
so $S(t) = T(t)$. Thus
\[ \lim_{n \to \infty} d(S(T(x_n)), T(S(x_n))) = \lim_{n \to \infty} d(S(t), T(t))
   = 0.
\]
Therefore, $S$ and $T$ are compatible.

We show compatible implies compatible of type C, as follows. 
Let $S$ and $T$ be compatible. Let $\{x_n\}_{n=1}^{\infty} \subset X$ 
satisfy~(\ref{repeatedCompatSeq}). By Proposition~\ref{eventuallyConst}, 
$S(x_n)=t=T(x_n)$ for almost all~$n$, and by compatibility, $S(t)=T(t)$. Therefore,
\[ \lim_{n \to \infty} d(S(T(x_n)), T(T(x_n))) =
   \lim_{n \to \infty} d(S(t), T(t)) = 0,
\]
so~(\ref{compatC1stIneq}) is satisfied, and
\[ \lim_{n \to \infty} d(T(S(x_n)), S(S(x_n))) = d(T(t), S(t)) = 0,
\]
so~(\ref{compatC2ndIneq}) is satisfied. Thus $S$ and $T$ are
compatible of type C.

We show compatible of type C implies compatible, as follows. Let $S$ and $T$ be
compatible of type C. Let $\{x_n\}_{n=1}^{\infty} \subset X$ 
satisfy~(\ref{repeatedCompatSeq}). By Proposition~\ref{eventuallyConst}, 
$S(x_n)=t=T(x_n)$ for almost all~$n$. From~(\ref{compatC1stIneq}) it follows that
\[ d(S(t), T(t)) \le 1/2~[d(S(t),S(t)) + d(S(t),S(t)) + d(S(t), T(t))],
\]
or
\[ d(S(t), T(t)) \le 1/2~[ 0 + 0 + d(S(t), T(t))],
\]
which implies 
\[ 0 = d(S(t),T(t)) = \lim_{n \to \infty} d(S(T(x_n)), T(S(x_n))).
\]
Therefore, $S$ and $T$ are compatible.
\end{proof}

\subsection{Fixed point assertions of~\cite{Jyoti-Rani-Rani17b}}
\label{JRRfpSubsection}
The following assertion appears as {\bf Theorem~3.1.1 of~\cite{Jyoti-Rani-Rani17b}} 
and as {\bf Theorem~4.12 of~\cite{EgeEtal}} (there is a minor difference between
these: ~\cite{Jyoti-Rani-Rani17b} requires $\mu \in (0,1/2)$ while 
\cite{EgeEtal} requires $\mu \in (0,1)$).
\begin{assert}
\label{JRassert1}
    Let $(X,d,\kappa)$ be a complete digital metric space. Let $S$ and $T$ be compatible self-maps on
    $X$. Suppose
    
    (i) $S(X) \subset T(X)$;
    
    (ii) $S$ or $T$ is continuous; and
    
    (iii) for all $x,y \in X$ and some $\mu \in (0,1/2)$,
    \[ d(Sx,Sy) \le \mu \max \{ d(Tx,Ty), d(Tx,Sx), d(Tx,Sy), d(Ty,Sx), d(Ty,Sy)\}.
    \]
    Then $S$ and $T$ have a unique common fixed point in~$X$.
\end{assert}

\begin{remarks}
The argument given as proof in~\cite{Jyoti-Rani-Rani17b} for this assertion 
clarifies that the continuity assumed is topological (the classical
$\varepsilon - \delta$ continuity), not digital. 
\end{remarks}

Further, Assertion~\ref{JRassert1} and the argument offered for its proof
in~\cite{Jyoti-Rani-Rani17b} are flawed as discussed below (this is the first
of several assertions with related flaws; we discuss these assertions together),
beginning at Remark~\ref{remark1}. Flaws in the treatment of Assertion~\ref{JRassert1}
in~\cite{EgeEtal} are discussed below, beginning at Remark~\ref{EgeEtal4.12}.

The following assertion appears as {\bf Theorem~3.2.1 of~\cite{Jyoti-Rani-Rani17b}}.
\begin{assert}
\label{JRassert2}
    Let $(X,d,\kappa)$ be a complete digital metric space. Let $S$ and $T$ be 
    weakly compatible self-maps on $X$. Suppose
    
    (i) $S(X) \subset T(X)$;
    
    (ii) $S(X)$ or $T(X)$ is complete; and
    
    (iii) for all $x,y \in X$ and some $\mu \in (0,1/2)$,
    \[ d(Sx,Sy) \le \mu \max \{ d(Tx,Ty), d(Tx,Sx), d(Tx,Sy), d(Ty,Sx), d(Ty,Sy) \}.
    \]
    Then $S$ and $T$ have a unique common fixed point in~$X$.
\end{assert}

However, this assertion and the argument offered for its proof are flawed as discussed below, beginning at Remark~\ref{remark1}.

The following is {\bf Theorem~3.3.2 of~\cite{Jyoti-Rani-Rani17b}}.
\begin{thm}
\label{JRassert3}
Let $(X,d,\kappa)$ be a digital metric space. Let $S,T: X \to X$ be weakly compatible maps
satisfying the following.

(i) For some $\mu \in (0,1)$ and all $x,y \in X$, $d(Sx,Sy) \le \mu d(Tx,Ty)$.

(ii) $S$ and $T$ satisfy property E.A.

(iii) $T(X)$ is a closed subspace of $X$.

Then $S$ and $T$ have a unique common fixed point in $X$.
\end{thm}

However, this result is limited, as discussed below, beginning at Remark~\ref{remark1}.

The following appears as {\bf Theorem~3.3.3 of~\cite{Jyoti-Rani-Rani17b}}.
\begin{assert}
\label{JRassert4}
 Let $(X,d,\kappa)$ be a complete digital metric space. Let $S$ and $T$ be 
    weakly compatible self-maps on $X$. Suppose
    
    (i) for all $x,y \in X$ and some $\mu \in (0,1/2)$,
    \[ d(Sx,Sy) \le \mu \max \{ d(Tx,Ty), d(Tx,Sx), d(Tx,Sy), d(Ty,Sx), d(Ty,Sy) \};
    \]
    
    (ii) $S$ and $T$ satisfy the property E.A.; and
    
    (iii) $T(X)$ is a closed subspace of $X$.
    
    Then $S$ and $T$ have a unique common fixed point in $X$.
\end{assert}
However, this assertion and the argument offered for its proof are flawed as discussed below, beginning at Remark~\ref{remark1}.

The following is {\bf Theorem~3.4.3 of~\cite{Jyoti-Rani-Rani17b}}.

\begin{thm}
\label{JR3.4.2}
Let $S$ and $T$ be weakly compatible self-maps on a digital metric space $(X,d,\kappa)$
satisfying

(i) for some $\mu \in (0,1)$ and all $x,y \in X$, $d(Sx,Sy) \le \mu d(Tx,Ty)$; and

(ii) the CLRT property.

Then $S$ and $T$ have a unique common fixed point in $X$.
\end{thm}

However, this result is quite limited, as discussed below, beginning at Remark~\ref{remark1}.

The following appears as {\bf Theorem~3.4.3 of~\cite{Jyoti-Rani-Rani17b}}.

\begin{assert}
\label{JR3.4.3}
Let $S$ and $T$ be weakly compatible self-maps on a digital metric space $(X,d,\kappa)$
satisfying

(i) for all $x,y \in X$ and some $\mu \in (0,1/2)$,
    \[ d(Sx,Sy) \le \mu \max \{ d(Tx,Ty), d(Tx,Sx), d(Tx,Sy), d(Ty,Sx), d(Ty,Sy) \} \mbox{; and}
    \]
    
(ii) the CLRT property.

Then $S$ and $T$ have a unique common fixed point in $X$.
\end{assert}

However, this assertion and the argument offered for its proof are flawed as discussed below.

\begin{remarks}
\label{remark1}
Several times in the arguments offered as proofs for
Assertions~\ref{JRassert1}, \ref{JRassert2}, \ref{JRassert4}, and \ref{JR3.4.3},
inequalities appear that seem to confuse ``$\min$" and ``$\max$". E.g., in the argument for
Assertion~\ref{JRassert1}, it is claimed that the right side of the inequality
\[ d(y_n,y_{n+1}) \le \mu \max
  \left \{ \begin{array}{c}
         d(y_{n-1},y_n), ~d(y_{n-1},y_n), ~d(y_{n-1},y_{n+1}), \\
         d(y_n,y_n), ~d(y_n,y_{n+1})
  \end{array} \right \}
\]
is less than or equal to $\mu d(y_{n-1},y_{n+1})$, which would follow if ``$\max$" were
replaced by ``$\min$".
Thus, these assertions as given in~\cite{Jyoti-Rani-Rani17b}
must be regarded as unproven.
\end{remarks}

\begin{remarks}
Further, suppose ``$\min$" is substituted for ``$\max$" so that (iii) in each of the
Assertions~\ref{JRassert1} and~\ref{JRassert2} and (i) in each of
Assertions~\ref{JRassert4} and~\ref{JR3.4.3} becomes
\begin{quote}
    for all $x,y \in X$ and some $\mu \in (0,1/2)$,
    \[ d(Sx,Sy) \le \mu \min \{ d(Tx,Ty), d(Tx,Sx), d(Tx,Sy), d(Ty,Sx), d(Ty,Sy)\}.
    \]
\end{quote}
Then for all $x,y \in X$, $d(Sx,Sy) \le \mu d(Tx,Ty)$.
If $T \in C(X,c_u)$, $d$ is an $\ell_p$ metric, and
$\mu < 1/u^{1/p}$, then by Proposition~\ref{shrinkage}, $S$ is constant. It would then follow from
compatibility (respectively, from weak compatibility) that $S$ and $T$ have a
unique fixed point coinciding with the value of $S$.
\end{remarks}

\begin{remarks}
Similarly, in Theorems~\ref{JRassert3} and~\ref{JR3.4.2}, if $T \in C(X,c_u)$, $d$ is an $\ell_p$ metric, and
$\mu < 1/u^{1/p}$, then by Proposition~\ref{shrinkage}, $S$ is constant. It would then follow from
compatibility (respectively, from weak compatibility) that $S$ and $T$ have a
unique fixed point coinciding with the value of $S$.
\end{remarks}

\subsection{Fixed point assertions of~\cite{Rani-Jyoti-Rani}}
\label{RJRfpSubsection}
The following is stated as {\bf Lemma~3.3.5 of~\cite{Rani-Jyoti-Rani}}.
\begin{assert}
\label{RJ3.3.5}
Let $S,T: (X,d,\kappa) \to (X,d,\kappa)$ be compatible.

1) If $S(t)=T(t)$ then $S(T(t)) = T(S(t))$.

2) Suppose $\lim_{n \to \infty} S(x_n) = \lim_{n \to \infty} T(x_n) = t \in X$.
\begin{quote}
          (a) If $S$ is continuous at $t$, $\lim_{n \to \infty} T(S(x_n)) = S(t)$.
          
          (b) If $S$ and $T$ are continuous at $t$, then $S(t)=T(t)$ and $S(T(t)) = T(S(t))$.
\end{quote}
\end{assert}
But the continuity used in the proof of this assertion is topological continuity,
not digital continuity. We observe that if $X$ is finite or
$d$ is an $\ell_p$ metric, then the assumption of continuity need not be stated, as every self-map
on $X$ is continuous in the topological sense, since $(X,d)$ is discrete.

The argument given as proof of this assertion in~\cite{Rani-Jyoti-Rani} depends on the principle 
that $a_n \to a_0$ implies $S(a_n) \to S(a_0)$ if $S$ is continuous at $a_0$, a valid principle
for topological continuity and also for digital continuity 
if $X$ is finite or $d$ is an $\ell_p$ metric and $\kappa$ is a $c_u$ adjacency,
but, as shown in Example~\ref{nonStdMetric}, not generally true for digital continuity. Thus the
assertion must be regarded as unproven.

We can modify this assertion as follows. Notice we do not use a continuity hypothesis, but 
for part 2) we assume $X$ is finite or $d$ is an $\ell_p$ metric.

\begin{lem}
\label{substLemma}
Let $S,T: (X,d,\kappa) \to (X,d,\kappa)$ be compatible.

1) If $S(t)=T(t)$ then $S(T(t)) = T(S(t))$.

2) Suppose $X$ is finite or $d$ is an $\ell_p$ metric.
If 
\[ \lim_{n \to \infty} S(x_n) = \lim_{n \to \infty} T(x_n) = t \in X,
\]
then $\lim_{n \to \infty} T(S(x_n)) = S(t)=T(t)$ and $S(T(t)) = T(S(t))$.
\end{lem}

\begin{proof}
We modify the argument of~\cite{Rani-Jyoti-Rani}.

Suppose $S(t)=T(t)$. Let $x_n=t$ for all $n \in \N$. Then $S(x_n)=T(x_n)=S(t)=T(t)$, so
$d(S(T(t)), T(S(t))) = d(S(T(x_n)), T(S(x_n))) \to_{n \to \infty} 0$ by compatibility. This
establishes 1).

Suppose $\lim_{n \to \infty} S(x_n) = \lim_{n \to \infty} T(x_n) = t \in X$.
Since we assume $X$ is finite or $d$ is an $\ell_p$ metric,
we have $S(x_n) = T(x_n) = t$ for almost all $n$.
Therefore, for almost all $n$, the triangle inequality and
compatibility give us
\[d(T(S(x_n)), T(t)) \le d(T(S(x_n)), S(T(x_n))) + d(S(T(x_n)), T(t))  \]
  \[ \to 0 + \lim_{n \to \infty} d(S(T(x_n)), T(S(x_n))) =0,
\]
so $\lim_{n \to \infty} T(S(x_n)) =T(t)$.
Since $X$ is finite or $d$ is an $\ell_p$ metric, by compatibility
we have
\[ d(S(t), T(t)) = \lim_{n \to \infty} d(S(t), T(S(x_n))) =
   \lim_{n \to \infty} d(S(T(x_n)), T(S(x_n))) = 0.
\]
Therefore, $S(t) = T(t)$ and by part 1),
$S(T(t))=T(S(t))$.
\end{proof}

The following is stated as {\bf Theorem 3.3.6 of~\cite{Rani-Jyoti-Rani}}.

\begin{assert}
\label{RJRcompatibleThm}
Let $S$ and $T$ be continuous compatible maps of a complete digital metric space
$(X,d,\kappa)$ to itself. Then $S$ and $T$ have a unique common fixed point in $X$ if
for some $\alpha \in (0,1)$. 
\begin{equation}
    \label{3.1.4condition}
    S(X) \subset T(X) \mbox{ and } d(S(x),S(y)) \le \alpha d(T(x), T(y)) \mbox{ for all } x,y \in X.
\end{equation}
\end{assert}

\begin{remarks}
The argument given as proof of this assertion in~{\rm \cite{Rani-Jyoti-Rani}} 
clarifies that the assumed continuity is topological, not digital;
the argument is also flawed by its reliance on Assertion~\ref{RJ3.3.5}, which we have seen
is not generally valid. Thus, the assertion must be regarded as unproven.
\end{remarks}

As above, we can drop the assumption of continuity from Assertion~\ref{RJRcompatibleThm}
if we assume $X$ is finite or $d$ is an $\ell_p$ metric, as shown in the following.

\begin{thm}
\label{correctedRJRcompatibleThm}
Let $S$ and $T$ be compatible maps of a digital metric space
$(X,d,\kappa)$ to itself, where $X$ is finite or $d$
is an $\ell_p$ metric. If $S$ and $T$ satisfy~(\ref{3.1.4condition}) for some $\alpha \in (0,1)$, then they have a unique common fixed point in~$X$.
\end{thm}

\begin{proof}
We use ideas from the analog in~\cite{Rani-Jyoti-Rani}.

Let $x_0 \in X$. Since $S(X) \subset T(X)$, we can let $x_1 \in X$ such that $T(x_1)=S(x_0)$,
and, inductively, $x_n \in X$ such that $T(x_n)=S(x_{n-1})$ for all $n \in \N$. Then for
all $n>0$,
\[ d(T(x_{n+1}), T(x_n)) = d(S(x_n), S(x_{n-1})) \le \alpha d(T(x_n), T(x_{n-1})).
\]
It follows that $d(T(x_{n+1}), T(x_n)) \le \alpha^n d(T(x_1), T(x_0))$.
By Theorem~\ref{HanTrivialization}, there exists $t \in X$ such that $T(x_n)=t$ for almost
all~$n$. Our choice of the sequence $x_n$ then implies $S(x_n)=t$ for almost all~$n$.
By Lemma~\ref{substLemma}, $S(t) = T(t)$ and $S(T(t)) = T(S(t))$. Then
\[ d(S(t), S(S(t))) \le \alpha d(T(t), T(S(t))) = \alpha d(S(t), S(T(t))) =
   \alpha d(S(t), S(S(t))),
\]
so $(1 - \alpha) d(S(t), S(S(t))) \le 0$. Therefore, $d(S(t), S(S(t)))=0$, so 
\[S(t)=S(S(t))=S(T(t)) = T(S(t)).\]
Thus $S(t)$ is a common fixed point of $S$ and $T$.

To show the uniqueness of $t$ as a common fixed point, suppose $S(x)=T(x)=x$ and $S(y)=T(y)=y$.
Then
\[ d(x,y) = d(S(x), S(y)) \le \alpha d(T(x), T(y)) = \alpha d(x,y),
\]
so $(1 - \alpha) d(x,y) \le 0$, so $x=y$.
\end{proof}

The following is stated as {\bf Theorem 3.4.3 of~\cite{Rani-Jyoti-Rani}}.

\begin{assert}
\label{RJRweakCompatibleThm}
Let $S$ and $T$ be weakly compatible maps of a complete digital metric space
$(X,d,\kappa)$ to itself. Then $S$ and $T$ have a unique common fixed point in $X$ if
either of $S(X)$ or $T(X)$ is complete, and for some $\alpha \in (0,1)$,
statement~{\rm (\ref{3.1.4condition})} is satisfied.
\end{assert}

\begin{remarks}
The argument given in~{\rm \cite{Rani-Jyoti-Rani}} as a proof for Assertion~\ref{RJRweakCompatibleThm}
defines a sequence $\{x_n\}_{n=1}^{\infty} \subset X$
such that $\lim_{n \to \infty} S(x_n) = \lim_{n \to \infty} T(x_n) = t \in X$. From this is claimed that a subsequence of
$\{x_n\}_{n=1}^{\infty}$ converges to a limit in $X$. How this is justified is unclear. Therefore,
Assertion~\ref{RJRweakCompatibleThm} as stated is unproven. If we additionally assume that $X$ is finite, then the
claim, that a subsequence of $\{x_n\}_{n=1}^{\infty}$ converges to a limit in $X$, is certainly justified.

The following is a version of Assertion~\ref{RJRweakCompatibleThm} with the additional hypothesis
that $X$ is finite. We have not stated an assumption of completeness, since a 
finite metric space must be complete.
\end{remarks}

\begin{thm}
\label{correctedRJRweakCompatibleThm}
Let $S$ and $T$ be weakly compatible maps of a digital metric space
$(X,d,\kappa)$ to itself, where $X$ is finite. Then $S$ and $T$ have a unique common fixed point in 
$X$ if for some $\alpha \in (0,1)$,~{\rm (\ref{3.1.4condition})} is satisfied.
\end{thm}

\begin{proof}
Since $X$ is finite, it follows from Theorem~\ref{manyEquivs} that $S$ and $T$ are compatible.
The assertion follows from Theorem~\ref{correctedRJRcompatibleThm}.
\end{proof}

Note also that Theorems~\ref{RJRcompatibleThm}, \ref{correctedRJRcompatibleThm}, and~\ref{correctedRJRweakCompatibleThm} are limited by Proposition~\ref{shrinkage}.


\subsection{Fixed point assertions of~\cite{EgeEtal}}
The following appears as {\bf Proposition~4.10 of~\cite{EgeEtal}}. Apparently, the
authors neglected to state a hypothesis that $S$ and $T$ are compatible; they
used this hypothesis in their ``proof", and with this hypothesis, the
desired conclusion is correctly reached.

\begin{assert}
\label{EgeEtal4.10}
Let $S,T \in C(X,\kappa)$ for a digital metric space $(X,d,\kappa)$.
If $S(t)=T(t)$ for some $t \in X$, then 
\[ S(T(t)) = T(S(t)) = S(S(t)) = T(T(t)).
\]
\end{assert}

As stated, this is incorrect, as shown by the following example.

\begin{exl}
Let $S,T: \N \to \N$ be the functions
\[ S(x)=2,~~~ T(x) = x+1.
\]
Then $S,T \in C(\N, c_1)$ and $S(1)=T(1)=2$, but
\[ S(T(1))= S(S(1)) = 2,~~~ T(S(1)) = T(T(1)) = 3.
\]
\end{exl}

Further, the argument of~\cite{EgeEtal} does not use the hypothesis of continuity. Thus, Assertion~\ref{EgeEtal4.10} should be stated as follows.

\begin{prop}
{\rm \cite{EgeEtal}}
\label{correctedEgeEtal4.10}
Let $S,T: X \to X$ for a digital metric space $(X,d,\kappa)$. Suppose $S$ and $T$ 
are compatible. If $S(t)=T(t)$ for some $t \in X$, then 
\[ S(T(t)) = T(S(t)) = S(S(t)) = T(T(t)).
\] 
\end{prop}

Note that Proposition~\ref{correctedEgeEtal4.10} does not
require that $X$ be finite or $d$ be an $\ell_p$ metric, 
hence generalizes Lemma~\ref{substLemma}.

The following appears as {\bf Proposition~4.11 of~\cite{EgeEtal}}.

\begin{assert}
\label{EgeEtal4.11}
Let $(X,d,\kappa)$ be a digital metric space and let 
$S,T \in C(X,\kappa)$. Suppose
$\lim_{n \to \infty} S(x_n)=\lim_{n \to \infty} T(x_n)=t \in X$. Then

(i) $\lim_{n \to \infty} T(S(x_n))=S(t)$;

(ii) $\lim_{n \to \infty} S(T(x_n))=T(t)$; and

(iii) $S(T(t)) = T(S(t))$ and $S(t) = T(t)$.
\end{assert}

The ``proof" of this assertion in~\cite{EgeEtal}
confuses topological and digital continuity.
The following shows that the assertion is not generally true.

\begin{exl}
Let $S,T: \N \cup \{0\} \to \N \cup \{0\}$ be the functions $S(x)=0$, $T(x)=x+1$.
Let $d$ be the metric of Example~\ref{nonStdMetric}. Clearly, 
$S,T \in C(\N \cup \{0\}, c_1)$,
and with respect to $d$, we have 
$\lim_{n \to \infty} S(x_n) = 0 = \lim_{n \to \infty} T(x_n)$. However, with
respect to $d$ we have

(i) $\lim_{n \to \infty} T(S(x_n))= T(0) = 1$, $S(0)=0$.

(ii) $\lim_{n \to \infty} S(T(x_n))= 0$, $T(0)=1$;

(iii) $S(T(0)) = 0$, $T(S(0)) = 1$, $S(0) = 0$, $T(0)=1$.
\end{exl}

\begin{remarks}
\label{EgeEtal4.12}
We have stated {\bf Theorem~4.12 of~\cite{EgeEtal}}
above as Assertion~\ref{JRassert1}. The argument for this assertion in~\cite{EgeEtal} is flawed
as follows.
\end{remarks}

The argument considers the case $x_n=x_{n+1}$ and reaches the statement
\[ d(T(x_n), T(x_{n+1})) = d(S(x_{n-1}), S(x_n)) \le \]
\[   \alpha \max \{d(T(x_{n-1}), T(x_n)), d(T(x_{n-1}), T(x_{n+1})), 
                 d(T(x_n), T(x_{n+1})) \}.
\]
This yields three cases, each of which is handled incorrectly:
\begin{enumerate}
    \item $d(T(x_n), T(x_{n+1})) \le \alpha d(T(x_{n-1}), T(x_n))$. Nothing further is
          stated about this case.
    \item $d(T(x_n), T(x_{n+1})) \le \alpha d(T(x_{n-1}), T(x_{n+1}))$. 
          The authors misstate this case as 
          $d(T(x_n), T(x_{n+1})) \le \alpha d(T(x_n), T(x_{n+1}))$ and propagate
          this error forward.
    \item $d(T(x_n), T(x_{n+1})) \le \alpha d(T(x_n), T(x_{n+1}))$. This implies $T(x_n)=T(x_{n+1})$, since $0 < \alpha < 1$,
          but the authors reach a slightly weaker conclusion differently. They reason that 
          $d(T(x_n), T(x_{n+1})) \le \alpha^n d(T(x_0), T(x_1))$,
          from an implied induction with the unjustified assumption that this
          case applies at every level of the induction.
\end{enumerate}
Later in the argument, the error of confusing topological 
and digital continuity also appears.

Therefore, we must consider Assertion~\ref{JRassert1} unproven.

The following is stated as {\bf Theorem~4.13 of~\cite{EgeEtal}}.

\begin{assert}
\label{EgeEtal4.13}
Let $S, T: (X,d,\kappa) \to (X,d,\kappa)$ be mappings that are
compatible of type A on a digital metric space, such that

(i) $S(X) \subset T(X)$;

(ii) $S$ or $T$ is $(\kappa,\kappa)$-continuous; and

(iii) for all $x,y \in X$ and some $\alpha \in (0,1)$,
    \[ d(Sx,Sy) \le \alpha \max \{ d(Tx,Ty), d(Tx,Sy), d(Ty,Sx), d(Tx,Sx), d(Ty,Sy)\}.
    \]
    Then $S$ and $T$ have a unique common fixed point in~$X$.
\end{assert}

However, the argument given in~\cite{EgeEtal} to prove this assertion relies
on Assertion~\ref{EgeEtal4.11}, which we have shown above is unproven.

The following is stated as {\bf Theorem~4.14 of~\cite{EgeEtal}}.

\begin{assert}
\label{EgeEtal4.14}
Let $S, T: (X,d,\kappa) \to (X,d,\kappa)$ be mappings that are
compatible of type B on a digital metric space, satisfying (i), (ii), and (iii)
of Assertion~\ref{EgeEtal4.13}.
    Then $S$ and $T$ have a unique common fixed point in~$X$.
\end{assert}

However, the argument given in~\cite{EgeEtal} to prove this assertion relies
on Assertion~\ref{JRassert1}, which we have shown above to be unproven.

The following is stated as {\bf Theorem~4.15 of~\cite{EgeEtal}}.

\begin{assert}
\label{EgeEtal4.15}
Let $S, T: (X,d,\kappa) \to (X,d,\kappa)$ be mappings that are
compatible of type C on a digital metric space, satisfying (i), (ii), and (iii)
of Assertion~\ref{EgeEtal4.13}.
    Then $S$ and $T$ have a unique common fixed point in~$X$.
\end{assert}

However, the argument given in~\cite{EgeEtal} to prove this assertion relies
on Assertion~\ref{JRassert1} and on Assertion~\ref{EgeEtal4.14},
which we have shown above to be unproven.

The following is stated as {\bf Theorem~4.16 of~\cite{EgeEtal}}.

\begin{assert}
\label{EgeEtal4.16}
Let $S, T: (X,d,\kappa) \to (X,d,\kappa)$ be mappings that are
compatible of type P on a digital metric space, satisfying (i), (ii), and (iii)
of Assertion~\ref{EgeEtal4.13}.
    Then $S$ and $T$ have a unique common fixed point in~$X$.
\end{assert}

However, the argument given in~\cite{EgeEtal} to prove this assertion relies
on Assertion~\ref{JRassert1} and on Assertion~\ref{EgeEtal4.14},
each of which we have shown above to be unproven.

\section{Commutative and weakly commutative maps}
The paper~\cite{Rani-Jyoti-Rani} discusses common fixed points for
commutative and weakly commutative maps on digital metric spaces.

\begin{definition}
\label{commute-def}
Functions $f,g: X \to X$ are {\em commutative} if $f \circ g(x) = g \circ f(x)$ for
all $x \in X$.
\end{definition}

\begin{prop}
{\rm \cite{Rani-Jyoti-Rani}}
\label{commuteThm}
Let $(X,d,\kappa)$ be a digital metric space. Let $T: X \to X$. Then $T$ has a fixed point
in $X$ if and only if there is a constant function $S: X \to X$ such that 
$S$ commutes with $T$.
\end{prop}

The following is Theorem~3.1.4 of~\cite{Rani-Jyoti-Rani}.

\begin{thm}
\label{RJ3.1.4}
Let $T$ be a continuous mapping of a complete digital metric space $(X,d,\kappa)$ into itself.
Then $T$ has a fixed point in $X$ if and only if there exists $\alpha \in (0,1)$ and a
mapping $S: X \to X$ that commutes with $T$ and satisfies~(\ref{3.1.4condition}). 
\end{thm}

We give a modified version of Theorem~\ref{RJ3.1.4} as follows.

\begin{thm}
\label{modifiedRJ3.1.4}
Let $T$ be a mapping of a digital metric space $(X,d,\kappa)$ into itself.
\begin{itemize}
    \item If $T$ has a fixed point in $X$, then there exists 
          $\alpha \in (0,1)$ and a mapping $S: X \to X$ that
          commutes with $T$ and satisfies~(\ref{3.1.4condition}).
    \item Suppose $X$ is finite or $d$ is an $\ell_p$-metric.
          If there exists $\alpha \in (0,1)$ and a mapping 
          $S: X \to X$ that commutes with $T$ and satisfies~(\ref{3.1.4condition}), then $T$ has a fixed
          point in $X$.
\end{itemize}
\end{thm}

\begin{proof}
It follows from Proposition~\ref{commuteThm} that if $T$ has
a fixed point, then there is a mapping $S: X \to X$ that 
commutes with $T$ and satisfies~(\ref{3.1.4condition}).

Suppose $X$ is finite or $d$ is an $\ell_p$-metric. Suppose
there exists $\alpha \in (0,1)$ and a mapping $S: X \to X$
that commutes with $T$ and satisfies~(\ref{3.1.4condition}).
Then $S$ and $T$ are compatible by
Proposition~\ref{commuteImpliesCompatible}. It follows from
Theorem~\ref{correctedRJRcompatibleThm} that $T$ has a fixed point.
\end{proof}

We will use the following.

\begin{exl}
{\rm \cite{BxSt19}}
\label{lessNotCoContinuous}
Let $X = \{p_1,p_2,p_3\} \subset \Z^5$, where
\[ p_1=(0,0,0,0,0),~~~p_2=(2,0,0,0,0),~~~p_3=(1,1,1,1,1).
\]
Let $d$ be the Manhattan metric and let $T: (X,c_5) \to (X,c_5)$ be defined by
$T(p_1)=T(p_2)=p_1$, $T(p_3)=p_2$. Clearly $T(X) \subset 1_X(X)$, $1_X \in C(X,c_5)$, and
for all $x,y \in X$ we have $d(T(x),T(y)) \le 2/5 \, d(1_X(x), 1_X(y))$.
However, $T \not \in C(X,c_5)$ since $p_2 \adj_{c_5} p_3$ and $T(p_2) \not \adj_{c_5} T(p_3)$.
\end{exl}

\begin{cor}
{\rm \cite{Rani-Jyoti-Rani}}
\label{commuteCor1}
Let $T$ and $S$ be commuting maps of a digital metric space $(X,d,\kappa)$ into itself.
Suppose $T$ is continuous and $S(X) \subset T(X)$. If there exists $\alpha \in (0,1)$
and $k \in \N$ such that $d(S^k(x), S^k(y)) \le \alpha d(T(x),T(y))$ for all $x,y \in X$,
then $T$ and $S$ have a common fixed point.
\end{cor}

\begin{remarks}
The continuity hypothesized for Corollary~\ref{commuteCor1} in~\cite{Rani-Jyoti-Rani} is
topological continuity, not digital continuity. The assumption is used in the proof to argue 
that~(\ref{3.1.4condition}) implies $S$ is (topologically) continuous. Note if $X$ is finite or $d$ is an $\ell_p$ metric, then
$(X,d)$ is a discrete topological space and therefore every self-map on $X$ is topologically continuous.
If we were to assume instead that $T$ is digitally continuous, it would not follow from~(\ref{3.1.4condition}) that $S$ is
digitally continuous, as shown by Example~\ref{lessNotCoContinuous}.
\end{remarks}

The following is a modified version of Corollary~\ref{commuteCor1}. In it, there is no
continuity assumption, but we assume that $X$ is finite or $d$ is an $\ell_p$ metric.

\begin{cor}
Let $T$ and $S$ be commuting maps of a digital metric space $(X,d,\kappa)$ into itself.
Suppose $X$ is finite or $d$ is an $\ell_p$ metric. Suppose $S(X) \subset T(X)$. If there exists $\alpha \in (0,1)$
and $k \in \N$ such that $d(S^k(x), S^k(y)) \le \alpha d(T(x),T(y))$ for all $x,y \in X$,
then $T$ and $S$ have a common fixed point.
\end{cor}

\begin{proof}
As above, we modify the analogous argument of~\cite{Rani-Jyoti-Rani}.

We see easily that $S^k$ commutes with $T$ and 
$S^k(X) \subset S(X) \subset T(X)$.
By Theorem~\ref{RJ3.1.4}, there is a unique $a \in X$ such that
$a = S^k(a) = T(a)$. Since $S$ and $T$ commute, we can apply $S$ to the above to get
\[ S(a) = S(S^k(a)) = S(T(a)) = T(S(a))\]
and, from the first equation in this chain, $S(a) = S^k(S(a))$,
so $S(a)$ is a common fixed point of $T$ and $S^k$. Since $a$ is unique as a common fixed
point of $T$ and $S^k$, we must have $a=S(a)=T(a)$.
\end{proof}

A function $T: X \to X$ on a digital metric space $(X,d,\kappa)$ is a 
{\em digital expansive mapping}~\cite{Jyoti-Rani18}
if for some $k>1$ and all $x,y \in X$, $d(T(x),T(y)) \ge k d(x,y)$. However, 
this definition is quite limited, as shown by the following, which combines
Theorems~4.8 and~4.9 of~\cite{BxSt19}.

\begin{thm}
\label{expansiveNotWork}
Let $(X,d,\kappa)$ be a digital metric space. Suppose there are points
$x_0,y_0 \in X$ such that
\[ d(x_0,y_0) \in \{\min\{d(x,y)\,|\,x,y \in X, x \neq y\},~
                    \max\{d(x,y)\,|\,x,y \in X, x \neq y\}\}.
\]
Then there is no $T: X \to X$ that is both onto and a digital expansive mapping.
\end{thm}

Note the hypothesis of Theorem~\ref{expansiveNotWork} is satisfied by every finite digital metric space.

The following appears as Corollary 3.1.6 of~\cite{Rani-Jyoti-Rani}.

\begin{assert}
\label{JRwrong3.1.6}
Let $n \in \N$, $K \in \R$, $K>1$. Let $S: X \to X$ be a $\kappa$-continuous onto mapping of
a complete digital metric space $(X,d,\kappa)$ such that $d(S^n(x), S^n(y)) \ge K d(x,y)$ for all
$x,y \in X$. Then $S$ has a unique fixed point.
\end{assert}

\begin{remarks}
Theorem~\ref{expansiveNotWork} shows that Assertion~\ref{JRwrong3.1.6} is 
vacuous for finite digital metric spaces, since $S$ being onto implies
$S^n$ is onto. Similarly, Assertion~\ref{JRwrong3.1.6} is vacuous whenever
$\kappa = c_1$, since $x \adj_{c_1} y$ implies $S^n(x) \adjeq_{c_1} S^n(y)$,
hence $d(S^n(x), S^n(y)) \le d(x,y)$.
\end{remarks}

We give a corrected and improved version of Assertion~\ref{JRwrong3.1.6} as Corollary~\ref{commuteCor2} below, 
by making the following changes. 
\begin{itemize}
    \item We do not require $S$ to be either continuous or onto, nor
          do we require completeness.
    \item We use $K \in (0,1)$ rather than $K > 1$.
    \item We use $d(S^n(x), S^n(y)) \le K d(x,y)$ instead of $d(S^n(x), S^n(y)) \ge K d(x,y)$.
\end{itemize}

\begin{cor}
\label{commuteCor2}
Let $n \in \N$ and let $K \in (0, 1)$. Let $S \in C(X,\kappa)$ for
a digital metric space $(X,d,\kappa)$ such that
$d(S^n(x), S^n(y)) \le K d(x,y)$ for all $x,y \in X$, then $S$ has a unique fixed point.
\end{cor}

\begin{proof}
Take $T = 1_X \in C(X,\kappa)$. Then this assertion follows from Corollary~\ref{commuteCor1}.
\end{proof}

We modify assumptions of the second bullet of
Theorem~\ref{modifiedRJ3.1.4} to obtain 
a similar result with a much shorter proof.

\begin{thm}
Let $(X,d,c_u)$ be a digital metric space, where $d$ is an $\ell_p$ metric and
$X$ is $c_u$-connected. Let $T \in C(X,c_u)$.
Suppose we have a function $S: X \to X$ such that $S$ commutes with $T$,
$S(X) \subset T(X)$, and for some $\alpha \in (0,1/u^{1/p})$ and all $x,y \in X$,
$d(S(x),S(y)) \le \alpha d(T(x),T(y))$. Then $S$ is constant, and
$S$ and $T$ have a unique common fixed point.
\end{thm}

\begin{proof}
It follows from Proposition~\ref{shrinkage} that $S$ is a constant function.
Since $S(X) \subset T(X)$, the value $x_0$ taken by $S$ is a member of $T(X)$, and
since $S$ commutes with $T$,
\[ T(x_0) = T(S(x_0)) = S(T(x_0)) = x_0 = S(x_0).
\]
Since $S$ is constant, $x_0$ is a unique common fixed point.
\end{proof}

The following is Theorem~3.2.3 of~\cite{Rani-Jyoti-Rani}.

\begin{thm}
\label{RJ3.2.3}
Let $T$ be a mapping of a complete digital metric space $(X,d,\kappa)$ into itself.
Then $T$ has a fixed point in $X$ if and only if there exists $\alpha \in (0,1)$
and a mapping $S: X \to X$ that commutes weakly with $T$ and satisfies~(\ref{3.1.4condition}).
Indeed $T$ and $S$ have a unique common fixed point if~(\ref{3.1.4condition}) holds.
\end{thm}

However, Theorem~\ref{RJ3.2.3} is limited by Proposition~\ref{shrinkage}, which gives
conditions implying that the function $S$ must be constant.

\section{Commuting maps}
The paper~\cite{EgeEtal} studies common fixed points for commuting maps.

The following appears as {\bf Theorem~3.2 of~\cite{EgeEtal}}.

\begin{assert}
\label{EgeEtal3.2}
Let $\emptyset \neq X \subset \Z^n$, $n \in \N$, and let $S$ and $T$ be commuting mappings of a complete
digital metric space $(X,d,\kappa)$ into itself such that

(i) $T(X) \subset S(X)$;

(ii) $S \in C(X,\kappa)$; and

(iii) for some $\alpha \in (0,1)$ and all $x,y \in X$, $d(T(x),T(y)) \le \alpha d(S(x), S(y))$.

Then $S$ and $T$ have a common fixed point in $X$.
\end{assert}

\begin{remarks}
The argument given as proof for Assertion~\ref{EgeEtal3.2} in~\cite{EgeEtal} claims that (ii) and (iii) imply
$T \in C(X,\kappa)$, but this is incorrect (another instance of confusing topological and digital continuity), as shown in Example~\ref{lessNotCoContinuous}.
\end{remarks}

However, if we add to Assertion~\ref{EgeEtal3.2} the hypothesis that $X$ is finite or $d$ is an $\ell_p$ metric 
(and therefore $(X,d)$ is a complete metric space), then we get the following corrected version of
Assertion~\ref{EgeEtal3.2}.

\begin{thm}
\label{correctedEgeEtal3.2}
Let $\emptyset \neq X \subset \Z^n$, $n \in \N$. Let $S$ and $T$ be commuting mappings of a
digital metric space $(X,d,\kappa)$ into itself such that

(i) $T(X) \subset S(X)$; and

(ii) for some $\alpha \in (0,1)$ and all $x,y \in X$, $d(T(x),T(y)) \le \alpha d(S(x), S(y))$.

If $X$ is finite or $d$ is an $\ell_p$ metric, then $S$ and $T$ have a common fixed point in $X$.
\end{thm}

\begin{proof}
By Proposition~\ref{commuteImpliesCompatible}, $S$ and $T$ are compatible. The result follows
from Theorem~\ref{correctedRJRcompatibleThm}.
\end{proof}

The following is presented as {\bf Corollary~3.3 of~\cite{EgeEtal}}.

\begin{assert}
\label{EgeEtal3.3}
Let $S$ and $T$ be commuting mappings of a complete digital metric space
$(X,d,\kappa)$ into itself such that

(i) $T(X) \subset S(X)$;

(ii) $S \in C(X,\kappa)$; and

(iii) for some $\alpha \in (0,1)$ and $k \in \N$ we have
\[ d(T^k(x), T^k(y)) \le \alpha d(S(x), S(y)) \mbox{ for all } x,y \in X.
\]
Then $S$ and $T$ have a unique common fixed point.
\end{assert}

However, the argument given in~\cite{EgeEtal} for this assertion depends
on Assertion~\ref{EgeEtal3.2}, shown above as unproven. Assertion~\ref{EgeEtal3.3} can
be modified as follows.

\begin{cor}
\label{correctedEgeEtal3.3}
Let $S$ and $T$ be commuting mappings of a digital metric space
$(X,d,\kappa)$ into itself such that

(i) $T(X) \subset S(X)$; and

(ii) for some $\alpha \in (0,1)$ and $k \in \N$ we have
\[ d(T^k(x), T^k(y)) \le \alpha d(S(x), S(y)) \mbox{ for all } x,y \in X.
\]
If $X$ is finite or $d$ is an $\ell_p$ metric, then $S$ and $T$ have a 
unique common fixed point.
\end{cor}

\begin{proof}
We use the analogous argument of~\cite{EgeEtal}.

Clearly, $T^k$ commutes with $S$ and $T^k(X) \subset T(X) \subset S(X)$. From
Theorem~\ref{correctedEgeEtal3.2}, there is a unique $a \in X$ such that
$a = S(a) = T^k(a)$. By applying the function $T$ and the commuting property, we have
\[ T(a) = T(S(a)) = S(T(a)) \mbox{ and } T(a) = T(T^k(a)) = T^k(T(a)),
\]
so $T(a)$ is a common fixed point of $S$ and $T^k$. But $a$ is the unique common fixed point
of $S$ and $T^k$, so we must have $a = T(a)$, and we have already observed that $a=S(a)$, so the
proof is complete.
\end{proof}

\begin{remarks}
Note that Theorem~\ref{correctedEgeEtal3.2} and Corollary~\ref{correctedEgeEtal3.3}
are limited by Proposition~\ref{shrinkage}.
\end{remarks}

\section{Further remarks}
We have discussed assertions that appeared 
in~\cite{Jyoti-Rani-Rani17b,Rani-Jyoti-Rani,EgeEtal}. 
We have discussed errors or corrections for some,
shown some to be limited or trivial, and offered 
improvements for some.

\end{document}